\documentclass[12pt]{article}
\usepackage[dvips]{graphicx}
\usepackage{amsmath}
\usepackage{amsfonts}
\usepackage{amssymb}

\providecommand{\U}[1]{\protect\rule{.1in}{.1in}}

\newtheorem{theorem}{Theorem}[section]

\newtheorem{corollary}[theorem]{Corollary}

\newtheorem{definition}[theorem]{Definition}

\newtheorem{lemma}[theorem]{Lemma}

\newtheorem{proposition}[theorem]{Proposition}
\newtheorem{remark}[theorem]{Remark}

\newenvironment{proof}[1][Proof]{\textbf{#1.} }{\ \rule{0.5em}{0.5em}}

\begin{document}

\title{Matching subspaces in a field extension: an update}
\date{}
\author{Shalom Eliahou\thanks{%
eliahou@lmpa.univ-littoral.fr} and C\'{e}dric Lecouvey\thanks{%
lecouvey@lmpa.univ-littoral.fr} \\
%EndAName
LMPA Joseph Liouville, FR CNRS 2956\\
ULCO, B.P. 699, F-62228 Calais cedex\\
Univ Lille Nord de France, F-59000 Lille, France}
\maketitle

\begin{abstract}
In this paper, we formulate and prove linear analogues of results concerning
matchings in groups. A matching in a group $G$ is a bijection $\varphi $
between two finite subsets $A,B$ of $G$ with the property, motivated by old
questions on symmetric tensors, that $a\varphi (a)\notin A$ for all $a\in A$%
. Necessary and sufficient conditions on $G$, ensuring the existence of
matchings under appropriate hypotheses, are known. Here we consider a
similar question in a linear setting. Given a skew field extension $K\subset
L$, where $K$ commutative and central in $L$, we introduce analogous notions
of matchings between finite-dimensional $K$-subspaces $A,B$ of $L$, and
obtain existence criteria similar to those in the group setting. Our tools
mix additive number theory, combinatorics and algebra. The present version corrects a slight gap in the statement of Theorem~\ref{thm:matching property} of the published version of this paper.
\end{abstract}


\section{Introduction}

\label{matchgroup}

Throughout the paper, we shall say \textit{field} for a skew field or
division ring, and \textit{commutative field} for a field in which the
product is commutative.

Let $G$ be a group, written multiplicatively. Let $A,B \subset G$ be
nonempty finite subsets of $G$. A \textit{matching} from $A$ to $B$ is a map 
$\varphi :A\rightarrow B$ which is bijective and satisfies the condition 
\begin{equation*}
a\varphi(a)\notin A 
\end{equation*}
for all $a\in A$. This notion was introduced in \cite{FanLos} by Fan and
Losonczy, who used matchings in $\mathbb{Z}^{n}$ as a tool for studying an
old problem of Wakeford concerning canonical forms for symmetric tensors 
\cite{Wak}. Obvious necessary conditions for the existence of a matching
from $A$ to $B$ are $|A|=|B|$ and $1\notin B$. The group $G$ is said to have
the \textit{matching property} if these conditions on $A,B$ suffice to
guarantee the existence of a matching from $A$ to $B$. What groups possess
the matching property, and when are there automatchings from $B$ to $B$? The
following answers were first obtained by Losonczy \cite{LOs} in the abelian
case, and then extended to arbitrary groups in \cite{EL1}.

\begin{theorem}
\label{TH_matchAB}Let $G$ be group. Then $G$ has the matching property if
and only if $G$ is torsion-free or cyclic of prime order.
\end{theorem}

\begin{theorem}
\label{TH_matchA}Let $G$ be a group. Let $B$ be a nonempty finite subset of $%
G$. Then there is a matching from $B$ to $B$ if and only if $1\notin B$.
\end{theorem}

Theorem~\ref{TH_matchAB} and~\ref{TH_matchA} were established using methods
and tools pertaining to additive number theory and combinatorics.
Specifically, the additive tools used are lower bounds on the size of the
product set 
\begin{equation*}
AB = \{ab \mid a \in A, \; b \in B\}
\end{equation*}
in $G$, and the main combinatorial tool is Hall's marriage theorem. See also 
\cite{Hamidoune} for more results on matchings in groups.

\medskip

Now, various additive theorems bounding $|AB|$ have recently been transposed
to a linear setting, in the following sense. Given a field extension $K
\subset L$ and finite-dimensional $K$-subspaces $A,B$ of $L$, analogous
lower bounds on the dimension of $\langle AB \rangle$ were established,
where $\langle AB \rangle$ is the $K$-subspace spanned by the product set $AB
$ in $L$. See \cite{Xian, Hou, EL2, ELK}. Suitable hypotheses on the
extension may be needed, such as commutativity or separability. Two main
results in \cite{EL2}, which play a key role here, only require $K$ to be
commutative and central in $L$.

The purpose of this paper is to show that Theorem~\ref{TH_matchAB} and~\ref%
{TH_matchA} also admit linear analogues in a field extension $K \subset L$.
As in \cite{EL2}, we only assume that $K$ is commutative and central in $L$.
In Section~\ref{defs}, we introduce a specific notion of matching bases of
finite-dimensional subspaces $A,B$ of $L$, and state the main results of the
paper, namely Theorem~\ref{thm:matching property} and \ref{thm:automatching}%
. They are analogous to Theorem~\ref{TH_matchAB} and~\ref{TH_matchA}, and
give existence criteria for such matchings. The possibility of matching a
given basis of $A$ to some basis of $B$ is reformulated in Section~\ref{dims}%
, in terms of suitable dimension estimates. In the process, we use a linear
version of Hall's marriage theorem, derived from a more general theorem of
Rado on the existence of independent transversals in matroids. Section~\ref%
{linearized} presents the linear versions in \cite{EL2} of results in
additive number theory, that will allow us to deal with the required
dimension estimates of the preceding section. This is achieved in Section~%
\ref{proofs}, where Theorem~\ref{thm:matching property} and \ref%
{thm:automatching} are finally proved. In the last section, we introduce and
study a related notion of strong matching between subspaces of $L$.

\section{Definitions and main results}

\label{defs}

Throughout the paper, we shall consider a field extension $K \subset L$,
where $K$ is commutative and \textit{central} in $L$, i.e. such that $%
\lambda x=x\lambda$ for all $\lambda \in K, x \in L$. Let $A,B \subset L$ be
finite-dimensional $K$-subspaces of $L$. Ideally, a matching from $A$ to $B$
would be an isomorphism $\varphi : A \rightarrow B$ such that $a\varphi(a)
\notin A$ for all non-zero $a \in A$. However, we need to introduce somewhat
subtler requirements in order to obtain existence criteria analogous to
those of Theorem~\ref{TH_matchAB} and~\ref{TH_matchA}.

\smallskip

To start with, observe that if $0 \not=a \in A$ and $b \in B$, then 
\begin{equation*}
ab \notin A \; \Longleftrightarrow \; b \notin a^{-1}A \cap B.
\end{equation*}
This motivates the use of the subspace $a^{-1}A \cap B$ of $B$ in the
definition of matched bases below.

\begin{definition}
\label{def matched} Let $A,B$ be $n$-dimensional $K$-subspaces of the field
extension $L$. Let $\mathcal{A}= \{a_1,\ldots,a_n\}$, $\mathcal{B}%
=\{b_1,\ldots,b_n\}$ be bases of $A,B$ respectively. We say that $\mathcal{A}
$ is matched to $\mathcal{B}$ if 
\begin{equation*}
a_i b \in A \; \Longrightarrow \; b \in \langle b_1, \dots, \widehat{b_i},
\dots, b_n \rangle 
\end{equation*}
for all $b \in B$ and all $i = 1, \ldots, n$, where $\langle b_1, \dots, 
\widehat{b_i}, \dots, b_n \rangle$ is the hyperplane of $B$ spanned by the
set $\mathcal{B}\setminus \{b_i\}$; equivalently, if 
\begin{equation}  \label{match}
a_i^{-1}A \cap B \; \subset \; \langle b_1, \dots, \widehat{b_i}, \dots, b_n
\rangle
\end{equation}
for all $i = 1, \ldots, n$.
\end{definition}

\begin{remark}
If $\mathcal{A}$ is matched to $\mathcal{B}$ in the above sense, then it
follows that 
\begin{equation*}
a_ib_i \notin A, 
\end{equation*}
and hence $a_ib_i \notin \mathcal{A}$, for all $i = 1,\dots, n$. In
particular, the map $a_i \mapsto b_i$ is a matching, in the group setting
sense, from $\mathcal{A}$ to $\mathcal{B}$ within the multiplicative group $%
L^\ast$.
\end{remark}

\medskip

Moreover, we now show that if $\mathcal{A}$ is matched to $\mathcal{B}$,
then $B$ cannot contain 1. This necessary condition exactly mirrors the
corresponding one in the group setting.

\begin{lemma}
\label{1 notin B} Let $A,B$ be $n$-dimensional $K$-subspaces of the field
extension $L$. If a basis $\mathcal{A}$ of $A$ can be matched to a basis $%
\mathcal{B}$ of $B$, then $1 \notin B$.
\end{lemma}

\begin{proof}
Let $\mathcal{A}=\{a_1,\dots, a_n\}$, $\mathcal{B}=\{b_1,\dots, b_n\}$ be
bases of $A,B$ respectively. Assume on the contrary that $1 \in B$. Then we
have 
\begin{equation*}
1 \; \in \; \bigcap_{i=1}^n (a_i^{-1}A \cap B). 
\end{equation*}
On the other hand, it is clear that 
\begin{equation*}
\bigcap_{i=1}^n \langle b_1, \dots, \widehat{b_i}, \dots, b_n \rangle \; =
\; \{0\}. 
\end{equation*}
Therefore, the inclusion $a_i^{-1}A \cap B \,\subset \, \langle b_1, \dots, 
\widehat{b_i}, \dots, b_n \rangle$ required in (\ref{match}) cannot hold for
all $i=1,\ldots, n$, and hence $\mathcal{A}$ cannot be matched to $\mathcal{B%
}$.
\end{proof}

\bigskip

With the notion of matched bases at hand, we now introduce that of matched $K
$-subspaces of $L$.

\begin{definition}
Let $A,B$ be $n$-dimensional $K$-subspaces in the field extension $L$. We
say that $A$ is matched to $B$ if every basis $\mathcal{A}$ of $A$ can be
matched to a basis $\mathcal{B}$ of $B$.
\end{definition}

By the above lemma, if $A$ is matched to $B$, then $1 \notin B$. Conversely,
is the condition $1 \notin B$ sufficient to guarantee that any subspace $A$
of the same dimension as $B$ is matched to $B$? We shall see that the answer
depends on properties of the field extension $K \subset L$.

\begin{definition}
Let $K \subset L$ be a field extension. We say that $L$ has the linear
matching property if, for every $n \ge 1$ and every $n$-dimensional
subspaces $A,B$ of $L$ with $1 \notin B$, the subspace $A$ is matched to $B$.
\end{definition}

We shall prove the following results in Section~\ref{proofs}.

\begin{theorem}
\label{thm:matching property}Let $K\subset L$ be a field extension, with $K$
commutative and central in $L$. Then $L$ has the linear matching property if
and only if $L$ contains no proper finite-dimensional extension over $K$.%
\footnote{Our original statement of Theorem~\ref{thm:matching property} in the published version of this paper [Journal of Algebra 324 (2010) 3420-3430] was incorrect, as pointed out to us by Professors Akbari and Aliabadi. It mistakenly stated that the linear matching property was equivalent to $L$ being either transcendental or an extension of prime degree over $K$, thereby missing all finite-dimensional extensions of non-prime degree having no proper intermediate extensions. (See Remark~\ref{Rem_Gap}).}
\end{theorem}

\begin{corollary}
Let $L$ be a commutative finite-dimensional Galois extension of the
(commutative) field $K$. Then $L$ has the linear matching property if and
only if $L$ is an extension of $K$ of prime degree.
\end{corollary}

In contrast, no special hypothesis on $L$ is needed to guarantee that any $n$%
-dimensional subspace $B$ avoiding 1 is matched to itself.

\begin{theorem}
\label{thm:automatching} Let $K \subset L$ be a field extension, with $K$
commutative and central in $L$. Let $B$ be a finite-dimensional subspace of $%
L$. Then $B$ is matched to itself if and only if $1\notin B$.
\end{theorem}

The proofs of these results involve delicate linearized versions, obtained
in \cite{EL2} and recalled in Section~\ref{linearized}, of classical
addition theorems due to Kemperman and Olson.

\section{Dimension criteria for matchable bases}

\label{dims} Let $K \subset L$ be a field extension, with $K$ commutative
and central in $L$, and let $A,B \subset L$ be $n$-dimensional $K$-subspaces
of $L$. In this section, we reformulate the property of a basis $\mathcal{A}$
of $A$ to be matchable to some basis of $B$, in terms of suitable dimension
estimates.

\begin{proposition}
\label{thm:dim} Let $\mathcal{A}=\{a_1,\ldots,a_n\}$ be a basis of $A$. Then 
$\mathcal{A}$ can be matched to a basis of $B$ if and only if, for all $J
\subset \{1,\ldots, n\}$, we have 
\begin{equation}  \label{eq:dim estimates}
\dim \bigcap_{i \in J} (a_i^{-1}A \cap B) \; \le \; n-|J|.
\end{equation}
\end{proposition}

For the proof of this equivalence in Section~\ref{proof thm:dim}, we shall
need a linear version of the classical marriage theorem of Hall \cite{Hall}.

\subsection{Free transversals}

\label{sec:Rado}

Let $E$ be a vector space over the field $K$ and let $\mathcal{E}%
=\{E_{1},E_{2},\dots,E_{n}\}$ be a collection of vector subspaces of $E$. A 
\textit{free transversal} for $\mathcal{E}$ is a free family of vectors $%
\{x_{1},\dots ,x_{n}\}$ in $E$ satisfying $x_{i}\in E_{i}$ for all $%
i=1,\dots,n.$ The following result of Rado \cite{Rado} gives necessary and
sufficient conditions for the existence of a free transversal for $\mathcal{E%
}$, very similar to those of Hall's marriage theorem. See also \cite{Mirsky,
Arocha, Moshonkin}.

\begin{theorem}
\label{TH_Lin8hall}Let $E$ be a vector space over $K$ and let $\mathcal{E}%
=\{E_{1},E_{2},\dots,E_{n}\}$ be a family of vector subspaces of $E$. Then $%
\mathcal{E}$ admits a free transversal if and only if 
\begin{equation}
\dim\underset{i\in J}{+}E_{i}\geq\left| J\right|  \label{LMC}
\end{equation}
for all $J\subset\{1,\dots,n\}.$
\end{theorem}

It is not too difficult to prove this result directly, by properly mimicking
a proof of its classical counterpart. In the above-mentioned paper of Rado,
Theorem~\ref{TH_Lin8hall} arises as a particular case of a more general
theorem concerning the existence of independent transversals in (possibly
infinite) matroids. A finite version would read as follows \cite[Chapter 12.2%
]{Oxley}.

\begin{theorem}
\label{Rado} Let $F$ be a finite set, let $\mathcal{F}=\{F_1,\ldots,F_n\}$
be a family of subsets of $F$, and let $M$ be a matroid over $F$ with rank
function $r$. Then the family $\mathcal{F}$ admits a transversal which is
independent in $M$ if and only if one has 
\begin{equation}  \label{eq:rado}
r(\bigcup_{i \in J} F_i) \ge |J|
\end{equation}
for all $J \subset \{1,\ldots,n\}$.
\end{theorem}

Theorem~\ref{TH_Lin8hall} can be derived from Theorem~\ref{Rado} as follows.
For each $1 \le i \le n$, pick a basis $F_i$ of the subspace $E_i$, let $%
\mathcal{F}=\{F_1,\ldots,F_n\}$, and set 
\begin{equation*}
F = \bigcup_{1 \le i \le n} F_i.
\end{equation*}
As matroid $M$ over $F$, we consider the collection of linearly independent
subsets in $F$, with rank function $r$ defined by 
\begin{equation*}
r(S) = \dim_K \langle S \rangle 
\end{equation*}
for all subsets $S \subset F$. Here, as earlier, $\langle S \rangle$ denotes
the subspace of $E$ spanned by $S$. We now apply Theorem~\ref{Rado} in this
situation. It is clear, from the definition of the rank function $r$, that 
\begin{equation*}
r(\bigcup_{i \in J} F_i) = \dim\underset{i\in J}{+}E_{i}. 
\end{equation*}
Thus, conditions~(\ref{LMC}) and (\ref{eq:rado}) are equivalent, and an
independent transversal for $\mathcal{F}$ given by Theorem~\ref{Rado} yields
a free transversal for the family $\mathcal{E}$.

\subsection{Proof of Proposition~\protect\ref{thm:dim}}

\label{proof thm:dim}

We shall use the following standard notation. We denote by 
\begin{equation*}
B^{\ast}=\{f:B\rightarrow K\mid f\text{ is linear}\} 
\end{equation*}
the \textit{dual} of $B$. Moreover, for any subspace $C\subset B$, we denote
by 
\begin{equation*}
C^{\perp}=\{f\in B^{\ast}\mid C\subset\ker f\} 
\end{equation*}
the \textit{orthogonal} of $C$ in $B^{\ast}$. Recall that $\dim C^{\perp} =
\dim B - \dim C$.

\smallskip

We now prove Proposition~\ref{thm:dim}, using Theorem~\ref{TH_Lin8hall} as a
key ingredient.

\medskip

\begin{proof}
$\Rightarrow$) Assume first that $\mathcal{A}$ is matched to the basis $%
\mathcal{B}=\{b_1,\ldots,b_n\}$ of $B$. It follows from condition~(\ref%
{match}) that 
\begin{equation*}
a_i^{-1}A \cap B \; \subset \; \langle b_1, \dots, \widehat{b_i}, \dots, b_n
\rangle 
\end{equation*}
for all $1 \le i \le n$. This implies, for any $J \subset \{1,\ldots,n\}$,
that 
\begin{equation*}
\bigcap_{i \in J} (a_i^{-1}A \cap B) \; \subset \; \bigcap_{i \in J} \langle
b_1, \dots, \widehat{b_i}, \dots, b_n \rangle \; = \; \langle \mathcal{B}
\setminus \{b_i \mid i \in J\} \rangle. 
\end{equation*}
It follows that $\dim \bigcap_{i \in J} (a_i^{-1}A \cap B) \; \le \; n - |J|$%
, as claimed.

\medskip

\noindent $\Leftarrow$) Assume now that, for all $J \subset \{1,\ldots,n\}$,
we have 
\begin{equation*}
\dim \bigcap_{i \in J} (a_i^{-1}A \cap B) \; \le \; n-|J|.
\end{equation*}
Taking the orthogonal in the dual space $B^\ast$, we get 
\begin{equation*}
\dim \left(\bigcap_{i \in J} (a_i^{-1}A \cap B)\right)^\perp \; \ge \; |J|, 
\end{equation*}
and hence 
\begin{equation*}
\dim \sum_{i \in J} (a_i^{-1}A \cap B)^\perp \; \ge \; |J|. 
\end{equation*}
By Theorem~\ref{TH_Lin8hall}, the linear version of Hall's theorem, the
above dimension bounds imply the existence of a free transversal 
\begin{equation*}
\varphi_1,\ldots,\varphi_n \; \in \; B^\ast 
\end{equation*}
for the system of subspaces $\{(a_i^{-1}A \cap B)^\perp\}_{1 \le i \le n}$.
In other words, we have 
\begin{equation}  \label{perp}
\varphi_i \in (a_i^{-1}A \cap B)^\perp
\end{equation}
for all $1 \le i \le n$, and $\{\varphi_1,\ldots,\varphi_n\}$ is free and
hence a basis of $B^\ast$.

Let $\mathcal{B}=\{b_1,\ldots,b_n\}$ be the unique basis of $B$ whose dual
basis $\mathcal{B}^\ast$ equals $\{\varphi_1,\ldots,\varphi_n\}$, i.e. such
that $b_i^\ast=\varphi_i$ for all $i$. By (\ref{perp}), we have 
\begin{equation*}
b_i^\ast \left(a_i^{-1}A \cap B\right) \; = \; \{0\}, 
\end{equation*}
whence $a_i^{-1}A \cap B \; \subset \; \langle b_1, \dots, \widehat{b_i},
\dots, b_n \rangle$ for all $i$, as desired.
\end{proof}

\section{Linear versions of some additive theorems}

\label{linearized} In order to exploit the equivalence given by Proposition~%
\ref{thm:dim}, we shall need tools to establish the required dimension
estimates (\ref{eq:dim estimates}). These tools will be conveniently
provided by two results in linearized additive number theory, both
established in \cite{EL2}.

\medskip

Our first tool is a linear version of a classical theorem of Kemperman \cite%
{Kem}.

\begin{theorem}
\label{th_AB}Let $K \subset L$ be a field extension, with $K$ commutative
and central in $L$. Let $A,B$ be finite-dimensional $K$-subspaces of $L$
such that $K\subset A\cap B$. Suppose there exist subspaces $\overline{A},%
\overline{B}\subset L$ such that 
\begin{equation*}
A=K\oplus\overline{A},\text{ }B=K\oplus\overline{B}\text{ and }K\cap (%
\overline{A}+\overline{B}+\langle\overline{A}\,\overline{B}\rangle )=\{0\}. 
\end{equation*}
Then 
\begin{equation*}
\dim\langle AB\rangle\geq\dim A+\dim B-1. 
\end{equation*}
\end{theorem}

For the proof of Theorem~\ref{thm:automatching} in Section~\ref{proofs:auto}%
, we shall actually use the following corollary.

\begin{corollary}
\label{corKem} Let $U,V$ be finite-dimensional $K$-subspaces of $L$. Assume
that $U$, $V$ and $UV$ are all three contained in a $K$-subspace $X$ of $L$
such that $K\cap X=\{0\}.$ Then 
\begin{equation*}
\dim X\geq\dim U+\dim V. 
\end{equation*}
\end{corollary}

\begin{proof}
Let $A=K\oplus U$, $B=K\oplus V$.\ Then $K\cap(U+V+\langle UV \rangle)=\{0\}.
$ Therefore Theorem~\ref{th_AB} applies, and gives 
\begin{equation*}
\dim\langle AB\rangle\geq\dim A+\dim B-1. 
\end{equation*}
Since $\langle AB\rangle=K\oplus(U+V+\langle UV \rangle)\subset K\oplus X$,
this gives $\dim\langle AB\rangle\leq\dim X+1$. With the equalities $\dim
A=\dim U+1$ and $\dim B=\dim V+1$, we then derive 
\begin{equation*}
\dim X\geq\dim\langle AB\rangle-1\geq\dim A+\dim B-2\geq\dim U+\dim V, 
\end{equation*}
as desired.
\end{proof}

\medskip

Our second promised tool from \cite{EL2} is a linear version of a classical
theorem of Olson \cite{Ols}. It will be used in the proof of Theorem~\ref%
{thm:matching property} in Section~\ref{proofs:match}.

\begin{theorem}
\label{TH-Ol8lin}Let $K \subset L$ be a field extension, with $K$
commutative and central in $L$. Let $A,B$ be finite-dimensional $K$%
-subspaces of $L$ distinct from $\{0\}$.\ Then there exist a $K$-subspace $S$
of $\langle AB\rangle$ and a subfield $M$ of $L$ such that

\begin{enumerate}
\item[(1)] $K\subset M\subset L,$

\item[(2)] $\dim S\geq\dim A+\dim B-\dim M$,

\item[(3)] $MS=S$ or $SM=S$.
\end{enumerate}
\end{theorem}

\section{Proofs of the main results}

\label{proofs}

Let again $K \subset L$ be a field extension, with $K$ commutative and
central in $L$. Let $A,B$ be $n$-dimensional $K$-subspaces of $L$. The above
linearized versions of additive theorems will allow us to fulfill, under
appropriate circumstances, the dimension estimates required by Proposition~%
\ref{thm:dim}, and thereby to prove Theorem~\ref{thm:matching property} and %
\ref{thm:automatching}.

\subsection{Proof of Theorem~\protect\ref{thm:automatching}}

\label{proofs:auto}

\begin{theorem}
\label{th-match1}Let $B$ be a finite-dimensional $K$-subspace of $L$. Then $B
$ is matched to itself if and only if $1\notin B$.
\end{theorem}

\begin{proof}
We already know from Lemma~\ref{1 notin B} that if $B$ contains 1, then $B$
cannot be matched to itself. Conversely, assume $1 \notin B$. Let $\mathcal{A%
}=\{a_1,\ldots, a_n\}$ be any basis of $B$. For $J \subset \{1,\ldots,n\}$,
denote 
\begin{equation*}
V_{J}=\bigcap_{i\in J}(a_i^{-1}B \cap B)=\{x\in B\mid a_i x\in B\text{ for
all }i\in J\}. 
\end{equation*}
It follows from Proposition~\ref{thm:dim} that $\mathcal{A}$ can be matched
to another basis of $B$ if and only if 
\begin{equation}  \label{eq:cond}
\dim V_J \; \le \; n-|J|
\end{equation}
for all $J \subset \{1,\ldots,n\}$. Denote now $B_J$ the subspace of $B$
spanned by the subset $\{a_i \mid i \in J\}$ of $\mathcal{A}$. Then we have $%
\dim B_J = |J|$, and 
\begin{equation*}
B_J V_J \subset B 
\end{equation*}
by construction. Since $1\notin B$, Corollary~\ref{corKem} applies, with $%
U,V,X$ standing for $B_J, V_{J}, B$ respectively. This gives 
\begin{equation*}
\dim B_J +\dim V_{J} \; \leq\; \dim B, 
\end{equation*}
i.e. exactly condition~(\ref{eq:cond}). By Proposition~\ref{thm:dim}, the
basis $\mathcal{A}$ can be matched to another basis of $B$. Therefore, the
space $B$ is matched to itself.
\end{proof}

\subsection{Proof of Theorem~\protect\ref{thm:matching property}}

\label{proofs:match}

We now turn to the characterization of all field extensions satisfying the
linear matching property.

\begin{theorem}
Let $K\subset $ be a field extension, with $K$ commutative and central in $L$%
. Then $L$ has the linear matching property if and only if $L$ contains no
proper finite-dimensional extension over $K$.
\end{theorem}

\begin{proof}
Assume first that $L$ is neither purely transcendental nor an extension of
prime degree. Then there is an element $a\in L$, of finite degree $n\geq2$
over $K$, such that $K(a)\varsubsetneqq L$. In particular, we have 
\begin{equation*}
K(a)=\langle 1,a,\ldots,a^{n-1} \rangle. 
\end{equation*}
Set $A= K(a)$. Let now $x\in L\setminus K(a)$, and set 
\begin{equation*}
B=\langle a,\ldots,a^{n-1},x \rangle. 
\end{equation*}
We claim that $A$ is not matched to $B$. Indeed, consider the basis $%
\mathcal{A}=\{1,a,\ldots,a^{n-1}\}$ of $A$. Since $A=K(a)$ is a subfield of $%
L,$ we have 
\begin{equation*}
a_{i}^{-1}A\cap B=A\cap B=\langle a,\ldots,a^{n-1} \rangle
\end{equation*}
for all $1 \le i \le n$. Therefore, the condition $\dim \bigcap_{i \in J}
(a_i^{-1}A \cap B) \; \le \; n-|J|$ of Proposition~\ref{thm:dim} does not
hold for $J = \{1,\ldots,n\}$, for instance. It follows that $\mathcal{A}$
cannot be matched to a basis of $B$.

\smallskip

Conversely, assume that the only finite-dimensional subfields of $L$ extending $K$ are $K$%
, and $L$ itself if it is finite-dimensional over $K$. The field $L=K$ contains no proper intermediate extension and vacuously
satisfies the linear matching property. Assume now $L\neq K$. Let $A,B$ be $n
$-dimensional $K$-subspaces of $L$ with $1\notin B$. Let $\mathcal{A}%
=\{a_{1},\ldots ,a_{n}\}$ be any basis of $A$. For any $J\subset \{1,\ldots
,n\}$, denote 
\begin{equation*}
V_{J}=\bigcap_{i\in J}(a_{i}^{-1}A\cap B)=\{x\in B\mid a_{i}x\in A\text{ for
all }i\in J\}.
\end{equation*}%
By Proposition~\ref{thm:dim} again, we know that $\mathcal{A}$ can be
matched to a basis of $B$ if and only if 
\begin{equation}
\dim V_{J}\;\leq \;n-|J|  \label{eq:cond2}
\end{equation}%
for all $J\subset \{1,\ldots ,n\}$. As earlier, denote $A_{J}$ the subspace
of $A$ spanned by the subset $\{a_{i}\mid i\in J\}$ of $\mathcal{A}$. Then
we have $\dim A_{J}=|J|$, and 
\begin{equation*}
A_{J}V_{J}\subset A.
\end{equation*}%
Set $W_{J}=K\oplus V_{J}$. We have $\dim W_{J}=\dim V_{J}+1$, and still $%
A_{J}W_{J}\subset A$ by construction. By Theorem \ref{TH-Ol8lin}, applied to
the subspaces $A_{J}$ and $W_{J}$, there is an intermediate field extension $%
K\subset M\subset L$ and a subspace $T\subset \langle A_{J}W_{J}\rangle $
such that 
\begin{equation}
\dim \langle A_{J}W_{J}\rangle \geq \dim A_{J}+\dim W_{J}-\dim M,
\label{Hall5}
\end{equation}%
and $MT=T$ or $TM=T$. We cannot have $M=L$, for otherwise $T=L$; but as $%
T\subset A_{J}W_{J}\subset A$, this would imply $A=L=B$, contradicting the
hypothesis $1\notin B$. It follows that $M=K$, and inequality (\ref{Hall5})
yields 
\begin{equation*}
\dim A\geq |J|+\dim V_{J},
\end{equation*}%
since $\langle A_{J}W_{J}\rangle \subset A$, $\dim W_{J}=\dim V_{J}+1$ and $%
\dim M=1$. Therefore conditions (\ref{eq:cond2}) are satisfied, implying
that $\mathcal{A}$ can be matched to a basis of $B$. It follows that $L$ has
the linear matching property.
\end{proof}

\bigskip 

\begin{corollary}
Let $L$ be a commutative finite-dimensional Galois extension of the
(commutative) field $K$. Then $L$ has the linear matching property if and
only if $L$ is an extension of $K$ of prime degree.
\end{corollary}

\begin{proof} Indeed, for a Galois extension of degree $n$, the intermediate extensions are in reversing bijection with the subgroups of the Galois group $G$ of order $n$. Thus, if $n$ is not a prime number, then $G$ will have proper subgroups $\{1\} \not= H \not= G$, thereby yielding proper intermediate extensions $L \supsetneq M \supsetneq K$.  
\end{proof}

\begin{remark}
\label{Rem_Gap} For any  non-prime integer $n>1$, there exists a field extension $K \subset L$ of characteristic $0$ and degree $n$ admitting no proper intermediate extension. It can be constructed as follows. 
Take $L=k(X_{1},\ldots ,X_{n})$ the field of rational functions in the
commutative variables $X_{1},\ldots ,X_{n}$ over an arbitrary field $k$ of
characteristic $0$. Set $S=L^{S_{n}}$ the subfield of $L$ of rational
symmetric functions. The field $L$ can be regarded as the decomposition
field of the polynomial%
\begin{equation*}
P(T)=\prod_{i=1}^{n}(T-X_{i})\in S[T].
\end{equation*}%
Therefore $L$ is a normal extension of $S$. Since $L$ is of characteristic $0$, it is a Galois extension of $S$. Its Galois group
is $S_{n}$ so $L$ has degree $n!$ over $S$. The subgroup $S_{n-1}\subset
S_{n}$ is maximal. Therefore, if we set $K=L^{S_{n-1}}$, the invariant subfield under the group $S_{n-1}$, then $L$ is an extension
of $K$ of degree $n$ with no proper intermediate extension.
\end{remark}

\subsection{A refinement}

Even if the extension $K \subset L$ does not satisfy the linear matching
property, it is still possible to match some subspaces $A,B$ of $L$ under
suitable circumstances. Let $n_0(K,L)$ denote the smallest degree of an
intermediate field extension $K \subsetneq M \subset L$. Thus $n_0(K,L) \ge 2
$, and $n_0(K,L) = \infty$ if the extension is purely transcendental.
Slightly adapting the proof of Theorem~\ref{thm:matching property} yields
the following result.

\begin{theorem}
\label{refin} Let $K \subset L$ be a field extension, with $K$ commutative
and central in $L$. Let $A,B \subset L$ be $n$-dimensional subspaces of $L$,
with $1 \notin B$ and $n < n_0(K,L)$. Then $A$ is matched to $B$.
\end{theorem}

\begin{proof}
Let $\mathcal{A}=\{a_1,\ldots, a_n\}$ be any basis of $A$. We proceed as in
the proof of Theorem~\ref{thm:matching property}, and use the same notation $%
V_J, A_J, W_J=K\oplus V_{J}$ for $J \subset \{1,\ldots,n\}$. In order to
ensure that $\mathcal{A}$ can be matched to a basis of $B$, it suffices to
check the condition 
\begin{equation}  \label{eq:cond3}
\dim V_J \; \le \; n-|J|
\end{equation}
for all $J \subset \{1,\ldots,n\}$. We have $A_J W_J \subset A$. By Theorem %
\ref{TH-Ol8lin} applied to $A_J$ and $W_J$, there is an intermediate field
extension $K\subset M \subset L$ and a subspace $T\subset \langle A_J W_J
\rangle$ such that 
\begin{equation}  \label{Hall6}
\dim \langle A_J W_J \rangle \ge \dim A_J+\dim W_{J}-\dim M,
\end{equation}
and $MT=T$ or $TM=T$. This means that $T$ is either a left of a right $M$%
-module, whence $\dim M$ divides $\dim T$. But since $T \subset \langle A_J
W_J \rangle \subset A$, it follows that $\dim M \le n$. Now, our assumption $%
n < n_0(K,L)$ implies $M = K$. Therefore, inequality (\ref{Hall6}) yields 
\begin{equation*}
\dim A\geq|J|+\dim V_{J}, 
\end{equation*}
since $\langle A_JW_{J}\rangle\subset A$, $\dim A_J = |J|$, $\dim W_{J}=\dim
V_{J}+1$ and $\dim M=1$. Thus (\ref{eq:cond3}) is satisfied, implying that $%
\mathcal{A}$ can be matched to a basis of $B$.
\end{proof}

\section{Strong matchings}

Here we turn to a related, but much stronger notion of matching between
subspaces. An existence criterion for such matchings is much easier to
establish, as we do now, independently of the preceding results.

Let $K\subset L$ be a field extension, and let $A,B$ be finite-dimensional $K
$-subspaces of $L$ distinct from $\{0\}$.\ 

\begin{definition}
A strong matching from $A$ to $B$ is a linear isomorphism $%
\varphi:A\rightarrow B$ such that any basis $\mathcal{A}$ of $A$ is matched
to the basis $\varphi(\mathcal{A})$ of $B$, in the sense of Definition~\ref%
{def matched}.
\end{definition}

We start with an easy equivalent reformulation of this notion, and then
proceed with the promised existence criterion for strong matchings.

\begin{lemma}
\label{Lem_strongmatch}Let $\varphi:A\rightarrow B$ be an isomorphism of $K$%
-vector spaces. The following two statements are equivalent.

\begin{enumerate}
\item The map $\varphi$ is a strong matching from $A$ to $B.$

\item For any $0 \not=a\in A$ and any subspace $H\subset A$ such that $%
A=\langle a \rangle\oplus H,$ we have $a^{-1}A \cap B\subset\varphi(H).$
\end{enumerate}
\end{lemma}

\begin{proof}
Assume $\varphi$ is a strong matching. Let $0 \not=a\in A$, and let $%
H\subset A$ be any $K$-subspace such that $A=\langle a \rangle\oplus H.$ Let 
$\{a_{2},\ldots,a_{n}\}$ be any basis of $H.$ Then $\mathcal{A}%
=\{a,a_{2},\ldots,a_{n}\}$ is a basis of $A.$ Hence $\mathcal{A}$ is matched
to $\varphi (\mathcal{A})=\{\varphi(a),\varphi(a_{2}),\ldots,\varphi(a_{n})%
\}.$ This implies that $a^{-1}A \cap B$ is a subspace of $\langle
\varphi(a_2), \ldots, \varphi(a_n) \rangle=\varphi(H).$

Conversely, assume statement 2 holds. Let $\mathcal{A}=\{a_{1},a_{2},%
\ldots,a_{n}\}$ be a basis of $A$. For any $i=1,\ldots,n,$ set $%
H_{i}=\langle a_1,\ldots, \widehat{a_i},\ldots, a_n \rangle$. Then we must
have $a_{i}^{-1}A \cap B\subset\langle \varphi(a_1),\ldots, \widehat{%
\varphi(a_i)},\ldots, \varphi(a_n) \rangle$. This proves that the basis $%
\mathcal{A}$ is matched to $\varphi (\mathcal{A)}$. Hence, the map $\varphi$
is a strong matching from $A$ to $B$.
\end{proof}

\begin{theorem}
Let $K\subset L$ be a field extension. Let $A,B$ be $n$-dimensional $K$%
-subspaces of $L$ distinct from $\{0\}$. There is a strong matching from $A$
to $B$ if and only if $AB\cap A=\{0\}.$ In this case, any isomorphism $%
\varphi:A\rightarrow B$ is a strong matching.
\end{theorem}

\begin{proof}
Assume $\varphi: A \rightarrow B$ is a strong matching. By Lemma \ref%
{Lem_strongmatch}, we obtain for any $0 \not= a\in A$: 
\begin{equation*}
a^{-1}A \cap B \; \subset\; {\bigcap\limits_{H}} \varphi(H)\; =\;
\varphi\left({\bigcap\limits_{H}}\, H \right), 
\end{equation*}
where $H$ ranges over all subspaces of $A$ such that $A=\langle a
\rangle\oplus H$. But of course, the intersection of all such subspaces $H$
is reduced to $\{0\}$. Hence, we have $a^{-1}A \cap B=\{0\}$ for any $0
\not=a\in A$. This means that $AB\cap A=\{0\}.$

Conversely, assume $AB\cap A=\{0\}$, and let $\varphi:A\rightarrow B$ be any
isomorphism. Then, for all $0 \not= a \in A$, we have $a^{-1}A \cap B=\{0\}$%
, whence $a^{-1}A \cap B \subset \varphi(H)$ for any subspace $H \subset A$.
It follows from Lemma \ref{Lem_strongmatch} that $\varphi$ is a strong
matching from $A$ to $B$, as claimed.
\end{proof}

\bigskip
\noindent
\textbf{Acknowledgement.} 
We are most grateful to Professors Akbari and Aliabadi for having pointed out to us
the now corrected gap in our published version of Theorem~\ref{thm:matching property}.


\begin{thebibliography}{99}
\bibitem{Arocha} \textsc{J.~L.~Arocha}, \textsc{B.~Llano} and \textsc{%
M.~Takane}, \textit{The theorem of Philip Hall for vector spaces}, An. Inst.
Mat. Univ. Nac. Aut\'onoma M\'exico \textbf{32} (1992), 1--8 (1993).

\bibitem{RD} \textsc{R.~Diestel,} \textit{Graph Theory}, Graduate Text in
Mathematics \textbf{173}, Springer-Verlag, New York, 1997.

\bibitem{EL1} \textsc{S.~Eliahou} and \textsc{C.~Lecouvey}, \textit{%
Matchings in arbitrary groups,} Adv. in Appl. Math. \textbf{40} (2008),
219--224.

\bibitem{EL2} \textsc{S.~Eliahou} and \textsc{C.~Lecouvey}, \textit{On
linear versions of some addition theorems}, Linear Multilinear Algebra 
\textbf{57} (2009), 759--775.

\bibitem{ELK} \textsc{S.~Eliahou}, \textsc{M.~Kervaire} and \textsc{%
C.~Lecouvey}, \textit{On the product of vector spaces in a commutative field
extension, } J. Number Theory \textbf{129} (2009), 339--348.

\bibitem{FanLos} \textsc{C.~K.~Fan} and \textsc{J.~Losonczy,} \textit{%
Matchings and canonical forms in symmetric tensors}, Adv. Math. \textbf{117}
(1996), 228--238.

\bibitem{Hall} \textsc{P.~Hall,} \textit{On representatives of subsets}, J.
London Math. Soc. \textbf{10} (1935), 26--30.

\bibitem{Hamidoune} \textsc{Y.~O.~Hamidoune,} \textit{On Group bijections $%
\phi $ with $\phi(B)=A$ and $\forall a\in B, a\phi(a) \notin A$}, preprint,
arXiv:0812.2522v1 [math.CO].

\bibitem{Hou} \textsc{X.D. Hou,} \textit{On a vector space analogue of
Kneser's theorem}, Linear Algebra Appl. \textbf{426} (2007), 214--227.

\bibitem{Xian} \textsc{X.~D.~Hou, K.~H.~Leung and Q.~Xiang,} \textit{A
generalization of an addition theorem of Kneser}, J. Number Theory \textbf{97%
} (2002), 1--9.

\bibitem{Kem} \textsc{J.~H.~B.~Kemperman,} \textit{On complexes in a
semigroup}, Indag.\ Math.\ \textbf{18} (1956), 247-254.

\bibitem{LOs} \textsc{J.~Losonczy,} \textit{On matchings in groups}, Adv. in
Appl. Math. \textbf{20} (1998), 385--391.

\bibitem{Mirsky} \textsc{L.~Mirsky} and \textsc{H.~Perfect,} \textit{Systems
of representatives}, J. Math. Anal. Appl. \textbf{15} (1966), 520--568.

\bibitem{Moshonkin} \textsc{A.~G.~Moshonkin}, \textit{Concerning Hall's
theorem}, Mathematics in St.Petersburg, 73--77, Amer. Math. Soc. Transl.
Ser. 2, \textbf{174}, Amer. Math. Soc., Providence, RI, 1996.

\bibitem{Nath} \textsc{M.~B.~Nathanson,} \textit{Additive Number Theory:
Inverse Problems and the Geometry of Sumsets}, Graduate Text in Mathematics 
\textbf{165}, Springer-Verlag, New York, 1996.

\bibitem{Ols} \textsc{J.~E.~Olson,} \textit{On the sum of two sets in a group%
}, J. Number Theory \textbf{18} (1984), 110--120.

\bibitem{Oxley} \textsc{J.~G.~Oxley,} \textit{Matroid theory}, Oxford
Science Publications. The Clarendon Press, Oxford University Press, New
York, 1992.

\bibitem{Rado} \textsc{R.~Rado,} \textit{A theorem on independence relations}%
, Quart. J. Math., Oxford Ser. \textbf{13} (1942), 83--89.

\bibitem{Wak} \textsc{E.~K.~Wakeford,} \textit{On canonical forms}, Proc.
London Math. Soc. \textbf{18} (1918-1919), 403--410.
\end{thebibliography}
\end{document}